\documentclass[a4paper,12pt]{article}
\usepackage{latexsym,amsmath,amsthm,amssymb}
\usepackage{a4wide}
\usepackage{hyperref}
\usepackage{marginnote}
\usepackage{color}
\usepackage{xcolor}
\usepackage{cite}
\hypersetup{
pdftitle={Adams hyperbolic}   
pdfauthor={Van Hoang Nguyen},
colorlinks = true,
linkcolor = magenta,
citecolor = blue,
}

\theoremstyle{plain}
\newtheorem{theorem}{Theorem}[section]

\newtheorem{proposition}[theorem]{Proposition}
\newtheorem{lemma}[theorem]{Lemma}

\theoremstyle{definition}

\theoremstyle{remark}

\renewcommand{\thefootnote}{\arabic{footnote}}

\def\R{\mathbb R}
\def\N{\mathbb N}

\def\al{\alpha}
\def\om{\omega}
\def\Om{\Omega}
\def\be{\beta}
\def\ga{\gamma}

\def\si{\sigma}
\def\lam{\lambda}
\def\vphi{\varphi}
\def\ep{\epsilon}
\def\na{\nabla}
\def\pa{\partial}
\def\la{\langle} 
\def\ra{\rangle} 
\def\lt{\left}
\def\rt{\right}

\def\H{\mathbb H}
\def\B{\mathbb B}
\def\leqs{<}
\def\geqs{>}

\numberwithin{equation}{section}

\title{The sharp Adams type inequalities in the hyperbolic spaces under the Lorentz-Sobolev norms}
\author{Van Hoang Nguyen
}

\begin{document}
\maketitle


\renewcommand{\thefootnote}{}

\footnote{Email:  \href{mailto: Van Hoang Nguyen <vanhoang0610@yahoo.com>}{vanhoang0610@yahoo.com}.} 

\footnote{2010 \emph{Mathematics Subject Classification\text}: 26D10, 46E35, 46E30, }

\footnote{\emph{Key words and phrases\text}: Adams inequality, improved Adams inequality, Hardy--Adams inequality Lorentz--Sobolev space, hyperbolic spaces, rearrangement argument.}

\renewcommand{\thefootnote}{\arabic{footnote}}
\setcounter{footnote}{0}

\begin{abstract}
Let $2\leq m \leqs n$ and $q \in (1,\infty)$, we denote by $W^mL^{\frac nm,q}(\mathbb H^n)$ the Lorentz--Sobolev space of order $m$ in the hyperbolic space $\mathbb H^n$. In this paper, we establish the following Adams inequality in the Lorentz--Sobolev space $W^m L^{\frac nm,q}(\mathbb H^n)$
\[
\sup_{u\in W^mL^{\frac nm,q}(\mathbb H^n),\, \|\nabla_g^m u\|_{\frac nm,q}\leq 1} \int_{\mathbb H^n} \Phi_{\frac nm,q}\big(\beta_{n,m}^{\frac q{q-1}} |u|^{\frac q{q-1}}\big) dV_g \leqs \infty
\]
for $q \in (1,\infty)$ if $m$ is even, and $q \in (1,n/m)$ if $m$ is odd, where $\beta_{n,m}^{q/(q-1)}$ is the sharp exponent in the Adams inequality under Lorentz--Sobolev norm in the Euclidean space. 
To our knowledge, much less is known about the Adams inequality under the Lorentz--Sobolev norm in the hyperbolic spaces. We also prove an improved Adams inequality under the Lorentz--Sobolev norm provided that $q\geq 2n/(n-1)$ if $m$ is even and $2n/(n-1) \leq q \leq \frac nm$ if $m$ is odd,
\[
\sup_{u\in W^mL^{\frac nm,q}(\mathbb H^n),\, \|\na_g^m u\|_{\frac nm,q}^q -\lam \|u\|_{\frac nm,q}^q \leq 1} \int_{\mathbb H^n} \Phi_{\frac nm,q}\big(\beta_{n,m}^{\frac q{q-1}} |u|^{\frac q{q-1}}\big) dV_g \leqs \infty
\]
for any $0\leqs \lambda \leqs C(n,m,n/m)^q$ where $C(n,m,n/m)^q$ is the sharp constant in the Lorentz--Poincar\'e inequality. Finally, we establish a Hardy--Adams inequality in the unit ball when $m\geq 3$, $n\geq 2m+1$ and $q \geq 2n/(n-1)$ if $m$ is even and $2n/(n-1) \leq q \leq n/m$ if $m$ is odd
\[
\sup_{u\in W^mL^{\frac nm,q}(\mathbb H^n),\, \|\na_g^m u\|_{\frac nm,q}^q -C(n,m,\frac nm)^q \|u\|_{\frac nm,q}^q \leq 1} \int_{\mathbb B^n} \exp\big(\beta_{n,m}^{\frac q{q-1}} |u|^{\frac q{q-1}}\big) dx \leqs \infty.
\]

\end{abstract}

\section{Introduction}
It is well-known that the Sobolev's embedding theorems play the important roles in the analysis, geometry, partial differential equations, etc. Let $m\geq 1$, we we traditionally use the notation
\[
\na^m = \begin{cases}
\Delta^{\frac m2} &\mbox{if $m$ is even,}\\
\na \Delta^{\frac{m-1}2} &\mbox{if $m$ is odd}
\end{cases}
\]
to denote the $m-$th derivatives. For a bounded domain $\Om\subset \R^n, n\geq 2$ and $1\leq p \leqs \infty$, we denote by $W^{m,p}_0(\Om)$ the usual Sobolev spaces which is the completion of $C_0^\infty(\Om)$ under the Dirichlet norm $\|\na^m u\|_{L^p(\Om)} = \Big(\int_\Om |\na^m u|^p dx \Big)^{\frac1p}$. The Sobolev inequality asserts that $W^{m,p}_0(\Om) \hookrightarrow L^q(\Om)$ for any $q \leq \frac{np}{n-mp}$ provided $mp \leqs n$. However, in the limits case $mp = n$ the embedding $W^{m,\frac nm}_0(\Om) \hookrightarrow L^\infty(\Om)$ fails. In this situation, the Moser--Trudinger inequality and Adams inequality are perfect replacements. The Moser--Trudinger inequality was proved independently by Yudovic \cite{Yudovic1961}, Pohozaev \cite{Pohozaev1965} and Trudinger \cite{Trudinger67}. This inequality was then sharpened by Moser \cite{Moser70} in the following form
\begin{equation}\label{eq:Moserineq}
\sup_{u\in W^{1,n}_0(\Om), \|\nabla u\|_{L^n(\Om)} \leq 1} \int_\Om e^{\alpha |u|^{\frac n{n-1}}} dx \leqs \infty
\end{equation}
for any $\al \leq \al_{n}: = n \om_{n-1}^{\frac 1{n-1}}$ where $\om_{n-1}$ denotes the surface area of the unit sphere in $\R^n$. Furthermore, the inequality \eqref{eq:Moserineq} is sharp in the sense that the supremum in \eqref{eq:Moserineq} will be infinite if $\al \geqs \al_n$.  The inequality \eqref{eq:Moserineq} was generalized to higher order Sobolev spaces $W^{m,\frac nm}_0(\Om)$ by Adams \cite{Adams} in the following form
\begin{equation}\label{eq:AMT}
\sup_{u \in W^{m,n}_0(\Om), \, \int_\Om |\na^m u|^{\frac nm} dx \leq 1} \int_\Om e^{\al |u|^{\frac n{n-m}}} dx \leqs \infty,
\end{equation}
for any 
\[
\al \leq \al_{n,m}: = \begin{cases}
\frac 1{\si_n}\Big(\frac{\pi^{n/2} 2^m \Gamma(\frac m2)}{\Gamma(\frac{n-m}2)}\Big)^{\frac n{n-m}} &\mbox{if $m$ is even},\\
\frac 1{\si_n}\Big(\frac{\pi^{n/2} 2^m \Gamma(\frac {m+1}2)}{\Gamma(\frac{n-m+1}2)}\Big)^{\frac n{n-m}} &\mbox{if $m$ is odd},
\end{cases}
\] 
where $\si_n = \om_{n-1}/n$ is the volume of the unit ball in $\R^n$. Moreover, if $\al \geqs \al_{n,m}$ then the supremum in \eqref{eq:AMT} becomes infinite though all integrals are still finite. 

The Moser-Trudinger inequality \eqref{eq:Moserineq} and Adams inequality \eqref{eq:AMT} play the role of the Sobolev embedding theorems in the limiting case $mp = n$. They have many applications to study the problems in analysis, geometry, partial differential equations, etc such as the Yamabe's equation, the $Q-$curvature equations, especially the problems in partial differential equations with exponential nonlinearity, etc. There have been many generalizations of the Moser--Trudinger inequality and Adams inequality in literature. For examples, the Moser--Trudinger inequality and Adams inequality were established in the Riemannian manifolds in \cite{YangSuKong,ManciniSandeep2010,AdimurthiTintarev2010,ManciniSandeepTintarev2013,Bertrand,Karmakar,LuTang2013,DongYang} and were established in the subRiemannian manifolds in \cite{CohnLu,CohnLu1,Balogh}. The singular version of the Moser--Trudinger inequality and Adams inequality was proved in \cite{AdimurthiSandeep2007,LamLusingular}. The Moser--Trudinger inequality and Adams inequality were extended to unbounded domains and whole spaces in \cite{Ruf2005,LiRuf2008,RufSani,AdimurthiYang2010,LamLuHei,Adachi00,LamLuAdams,LamLunew}, and to fractional order Sobolev spaces in \cite{Martinazzi,FM1,FM2}. The improved version of the Moser--Trudinger inequality and Adams inequality were given in \cite{AdimurthiDruet2004,Tintarev2014,WangYe2012,Nguyenimproved,LuYangAiM,Nguyen4,delaTorre,Mancini,Yangjfa,DOO,NguyenCCM,LuZhu,LuYangHA,LiLuYang}. An interesting question concerning to the Moser--Trudinger inequality and Adams inequality is whether or not the extremal functions exist. For this interesting topic, the reader may consult the papers \cite{Carleson86,Flucher92,Lin96,Ruf2005,LiRuf2008,Chen,LiYang,LuZhu,NguyenCCM,LuYangAiM,Nguyen4} and many other papers. 

Another generalization of the Moser--Trudinger inequality and Adams inequality is to establish the inequalities of same type in the Lorentz--Sobolev spaces. The Moser--Trudinger inequality and the Adams inequality in the Lorentz spaces was established by Alvino, Ferone and Trombetti \cite{Alvino1996} and Alberico \cite{Alberico} in the following form
\begin{equation}\label{eq:AMTLorentz}
\sup_{u\in W^m L^{\frac nm,q}(\Om), \, \|\na^m u\|_{\frac nm,q} \leq 1} \int_{\Om} e^{\al |u|^{\frac q{q-1}}} dx < \infty
\end{equation}
for any $\al \leq \beta_{n,m}^{\frac q{q-1}}$ with
\begin{equation*}
\beta_{n,m} = 
\begin{cases}
\frac{\pi^{n/2} 2^m \Gamma(\frac m2)}{\si_n^{(n-m)/n} \Gamma(\frac{n-m}2)}&\mbox{if $m$ is even,}\\
\frac{\pi^{n/2} 2^m \Gamma(\frac {m+1}2)}{\si_n^{(n-m)/n} \Gamma(\frac{n-m+1}n)}&\mbox{if $m$ is odd.}
\end{cases}
\end{equation*}
The constant $\beta_{n,m}$ is sharp in \eqref{eq:AMTLorentz} in the sense that the supremum will become infinite if $\al > \beta_{n,m}^{\frac q{q-1}}$. For unbounded domains in $\R^n$, the Moser--Trudinger inequality was proved by Cassani and Tarsi \cite{CassaniTarsi2009} (see Theorem $1$ and Theorem $2$ in \cite{CassaniTarsi2009}). In \cite{LuTang2016}, Lu and Tang proved several sharp singular Moser--Trudinger inequalities in the Lorentz--Sobolev spaces which generalize the results in \cite{Alvino1996,CassaniTarsi2009} to the singular weights. The singular Adams type inequalities in the Lorentz--Sobolev spaces were studied by the author in \cite{NguyenLorentz}.

The motivation of this paper is to study the Adams inequalities in the hyperbolic spaces under the Lorentz--Sobolev norm. For $n\geq 2$, let us denote by $\H^n$ the hyperbolic space of dimension $n$, i.e., a complete, simply connected, $n-$dimensional Riemmanian manifold having constant sectional curvature $-1$. The aim in this paper is to generalize the main results obtained by the author in \cite{Nguyen2020a} to the higher order Lorentz--Sobolev spaces in $\H^n$. Before stating our results, let us fix some notation. Let $V_g, \na_g$ and $\Delta_g$ denote the volume element, the hyperbolic gradient and the Laplace--Beltrami operator in $\H^n$ with respect to the metric $g$ respectively. For higher order derivatives, we shall adopt the following convention
\[
\na_g^m \cdot = \begin{cases}
\Delta_g^{\frac m2} \cdot &\mbox{if $m$ is even,}\\
\na_g (\Delta_g^{\frac{m-1}2} \cdot) &\mbox{if $m$ is odd.}
\end{cases}
\]
Furthermore, for simplicity, we write $|\na^m_g \cdot|$ instead of $|\na_g^m \cdot|_g$ when $m$ is odd if no confusion occurs. For $1\leq p, q\leqs \infty$, we denote by $L^{p,q}(\H^n)$ the Lorentz space in $\H^n$ and by $\|\cdot\|_{p,q}$ the Lorentz quasi-norm in $L^{p,q}(\H^n)$. When $p=q$, $\|\cdot\|_{p,p}$ is replaced by $\|\cdot\|_p$ the Lebesgue $L_p-$norm in $\H^n$, i.e., $\|f\|_p = (\int_{\H^n} |f|^p dV_g)^{\frac1p}$ for a measurable function $f$ on $\H^n$. The Lorentz--Sobolev space $W^m L^{p,q}(\H^n)$ is defined as the completion of $C_0^\infty(\H^n)$ under the Lorentz quasi-norm $\|\na_g^m u\|_{p,q}:=\| |\na_g^m u| \|_{p,q}$. In \cite{Nguyen2020a,Nguyen2020b}, the author proved the following Poincar\'e inequality in $W^1 L^{p,q}(\H^n)$
\begin{equation}\label{eq:Poincare}
\|\na_g^m u\|_{p,q}^q \geq C(n,m,p)^q \|u\|_{p,q}^q,\quad\forall\, u\in W^m L^{p,q}(\H^n).
\end{equation}
provided $1\leqs q \leq p$ if $m$ is odd and for any $1\leqs q \leqs \infty$ if $m$ is even, where
\begin{equation*}
C(n,m,p) = \begin{cases}
(\frac{(n-1)^2}{pp'})^{\frac m2} &\mbox{if $m$ is even,}\\
\frac {n-1}p (\frac{(n-1)^2}{pp'})^{\frac {m-1}2}&\mbox{if $m$ is odd,}
\end{cases}
\end{equation*}
with $p' = p/(p-1)$. Furthermore, the constant $C(n,m,p)^q$ in \eqref{eq:Poincare} is the best possible and is never attained. The inequality \eqref{eq:Poincare} generalizes the result in \cite{NgoNguyenAMV} to the setting of Lorentz--Sobolev space.

The Moser--Trudinger inequality in the hyperbolic spaces was firstly proved by Mancini and Sandeep \cite{ManciniSandeep2010} in the dimension $n =2$ (another proof of this result was given by Adimurthi and Tintarev \cite{AdimurthiTintarev2010}) and by Mancini, Sandeep and Tintarev \cite{ManciniSandeepTintarev2013} in higher dimension $n\geq 3$ (see \cite{FontanaMorpurgo2020} for an alternative proof)
\begin{equation}\label{eq:MThyperbolic}
\sup_{u\in W^{1,n}(\H^n),\, \int_{\H^n} |\na_g u|_g^n dV_g \leq 1} \int_{\H^n} \Phi(\al_n |u|^{\frac n{n-1}}) dV_g < \infty,
\end{equation}
where $\Phi(t) = e^t -\sum_{j=0}^{n-2} \frac{t^j}{j!}$. Lu and Tang \cite{LuTang2013} also established the sharp singular Moser--Trudinger inequality under the conditions $\|\na u\|_{L^n(\H^n)}^n + \tau \|u\|_{L^n(\H^n)}^n \leq 1$ for any $\tau >0$ (see Theorem $1.4$ in \cite{LuTang2013}). In \cite{NguyenMT2018}, the author improves the inequality \eqref{eq:MThyperbolic} by proving the following inequality 
\begin{equation}\label{eq:NguyenMT}
\sup_{u\in W^{1,n}(\H^n),\, \int_{\H^n} |\na_g u|_g^n dV_g - \lam \int_{\H^n} |u|^n dV_g \leq 1} \int_{\H^n} \Phi(\al_n |u|^{\frac n{n-1}}) dV_g < \infty,
\end{equation}
for any $\lambda < (\frac{n-1}n)^n$. The Adams inequality in the hyperbolic spaces were proved by Karmakar and Sandeep \cite{Karmakar} in the following form
\begin{equation*}
\sup_{u\in C_0^\infty(\H^{2n} \int_{\H^{2n}} P_nu \cdot u dV_g \leq 1} \int_{\H^{2n}} \Big(e^{ \al_{2n,n} u^2} -1\Big) dV_g < \infty.
\end{equation*}
where $P_k$ is the GJMS operator on the hyperbolic spaces $\H^{2n}$, i.e., $P_1 = -\Delta_g -n(n-1)$ and 
\[
P_k = P_1(P_1+2)\cdots (P_1 + k(k-1)),\quad k\geq 2.
\]
In recent paper, Fontana and Morpurgo \cite{FM2} established the following Adams inequality in the hyperbolic spaces $\H^n$,
\begin{equation}\label{eq:FM}
\sup_{u\in W^{m,\frac nm}(\H^n), \int_{\H^n} |\na_g^m u|^{\frac nm} dV_g \leq 1} \int_{\H^n} \Phi_{\frac nm}(\al_{n,m} |u|^{\frac n{n-m}}) dV_g < \infty
\end{equation}
where
\[
\Phi_{\frac nm}(t) = e^{t} -\sum_{j=0}^{j_{\frac nm} -2} \frac{t^j}{j!}, \quad\text{\rm and }\quad j_{\frac nm} = \min\{j\, :\, j \geq \frac nm\} \geq \frac nm.
\]
In \cite{NgoNguyenRMI}, Ngo and the author proved several Adams type inequalities in the hyperbolic spaces.

To our knowledge, much less is known about the Trudinger--Moser inequality and Adams inequality under the Lorentz--Sobolev norm on complete noncompact Riemannian manifolds except Euclidean spaces. Recently, Yang and Li \cite{YangLi2019} proves a sharp Moser--Trudinger inequality in the Lorentz--Sobolev spaces defined in the hyperbolic spaces. More precisely, their result (\cite[Theorem $1.6$]{YangLi2019}) states that for $1\leqs q \leqs \infty$  it holds
\begin{equation*}
\sup_{u\in W^1L^{n,q}(\H^n),\, \|\na_g u\|_{n,q} \leq 1} \int_{\H^n} \Phi_{n,q}(\al_{n,q} |u|^{\frac q{q-1}}) dV_g \leqs \infty,
\end{equation*}
where
\[
\Phi_{a,q}(t) =e^t - \sum_{j=0}^{j_{a,q} -2} \frac{t^j}{j!},\quad \text{\rm where}\,\, j_{a,q} = \min\{j\in \N\, :\, j \geqs 1+ a(q-1)/q\},
\]
with $a \geqs 1$.

The first aim in this paper is to establish the sharp Adams inequality in the hyperbolic spaces under the Lorentz--Sobolev norm which generalize the result of Yang and Li to higher order derivatives. Our fist result in this paper reads as follows.

\begin{theorem}\label{MAINI}
Let $n\geqs m \geq 2$ and $q \in (1,\infty)$. Then it holds
\begin{equation}\label{eq:AdamsLorentz}
\sup_{u\in W^mL^{\frac nm,q}(\H^n),\, \|\na_g^m u\|_{\frac nm,q}\leq 1} \int_{\H^n} \Phi_{\frac nm,q}\big(\beta_{n,m}^{\frac q{q-1}} |u|^{\frac q{q-1}}\big) dV_g \leqs \infty,
\end{equation}
for any $q \in (1,\infty)$ if $m$ is even, or $1\leqs q \leq \frac nm$ if $m$ is odd. Futhermore, the constant $\beta_{n,m}^{\frac q{q-1}}$ is sharp in the sense that the supremum in \eqref{eq:AdamsLorentz} will become infinite if $\beta_{n,m}^{\frac q{q-1}}$ is replaced by any larger constant.
\end{theorem}
Let us make some comments on Theorem \ref{MAINI}. When $q =\frac nm$, we obtain the inequality \eqref{eq:FM} of Fontana and Morpurgo from Theorem \ref{MAINI}. However, our approach is completely different with the one of Fontana and Morpurgo. Notice that in the case that $m$ is odd, we need an extra assumption Notice $q \leq \frac nm$ comparing with case that $m$ is even. This extra condition is a technical condition in our approach for which we can apply the P\'olya--Szeg\"o principle in the hyperbolic space (see Theorem \ref{PS} below). This principle was proved by the author in \cite{Nguyen2020a} which generalizes the classical P\'olya--Szeg\"o principle in Euclidean space to the hyperbolic space. Note that when $m=1$, the extra condition is not need by the result of Yang and Li \cite{YangLi2019}. The approach of Yang and Li is based on an representation formula for function via Green's function of the Laplace-Beltrami $-\Delta_g$ (similar with the one of Fontana and Morpurgo \cite{FM2}). Hence, we believe that the extra condition $q \leq \frac nm$ is superfluous when $m > 1$ is odd. One reasonable approach is to follow the one of Fontana and Morpurgo by using the representation formulas and estimates in \cite[Section $5$]{FM2}. This problem is left for interesting reader. 

Next, we aim to improve the Lorentz--Adams inequality in Theorem \ref{MAINI} in spirit of \eqref{eq:NguyenMT}. In the case $m=1$, an analogue of \eqref{eq:NguyenMT} under Lorentz--Sobolev norm was obtained by the author in \cite[Theorem $1.3$]{Nguyen2020a}. The result for $m > 1$ is given in the following theorem.

\begin{theorem}\label{MAINII}
Let $n > m\geq 2$ and $q \geq \frac{2n}{n-1}$. Suppose in addition that $q \leq \frac nm$ if $m$ is odd. Then we have
\begin{equation}\label{eq:improvedAL}
\sup_{u\in W^mL^{\frac nm,q}(\H^n),\, \|\na_g^m u\|_{\frac nm,q}^q -\lam \|u\|_{\frac nm,q}^q \leq 1} \int_{\H^n} \Phi_{\frac nm,q}\big(\beta_{n,m}^{\frac q{q-1}} |u|^{\frac q{q-1}}\big) dV_g \leqs \infty.
\end{equation}
for any $ \lam \leqs C(n,m,\frac nm)^q$.
\end{theorem}
Obviously, Theorem \ref{MAINII} is stronger than Theorem \ref{MAINI}. The extra condition $q \geq \frac{2n}{n-1}$ in Theorem \ref{MAINII} is to apply a crucial point-wise estimate in \cite[Lemma $2.1$]{NguyenPS2018}. Theorem \ref{MAINII} is proved by using iteration method and some estimates in \cite{Nguyen2020b} which we will recall in Section \S2 below.

The Hardy--Moser--Trudinger inequality was proved by Wang and Ye (see \cite{WangYe2012}) in dimension $2$
\begin{equation}\label{eq:WangYe}
\sup_{u \in W^{1,2}_0(\B^2), \int_{\B^2} |\na u|^2 dx - \int_{\B^2} \frac{u^2}{(1-|x|^2)^2} dx \leq 1} \int_{\B^2} e^{4\pi u^2} dx < \infty.
\end{equation}
The inequality \eqref{eq:WangYe} is stronger than the classical Moser--Trudinger inequality in $\B^2$. It connects both the sharp Moser--Trudinger inequality in $\B^2$ and the sharp Hardy inequality in $\B^2$
\[
\int_{\B^2} |\na u|^2 dx \geq  \int_{\B^2} \frac{u^2}{(1-|x|^2)^2} dx, \quad u \in W^{1,2}_0(\B^2).
\]
The higher dimensional version of \eqref{eq:WangYe} was recently established by the author \cite{NguyenHMT}
\[
\sup_{u \in W^{1,n}_0(\B^n), \int_{\B^n} |\na u|^n dx - \lt(\frac{2(n-1)}n\rt)^n\int_{\B^n} \frac{|u|^n}{(1-|x|^2)^n} dx \leq 1} \int_{\B^2} e^{\al_n |u|^{\frac n{n-1}}} dx < \infty.
\]
For higher order derivatives, the sharp Hardy--Adams inequality was proved by Lu and Yang \cite{LuYangHA} in dimension $4$ and by Li, Lu and Yang \cite{LiLuYang} in any even dimension. The approach in \cite{LuYangHA,LiLuYang} relies heavily on the Hilbertian structure of the space $W^{\frac n2,2}_0(\B^n)$ with $n$ even for which the Fourier analysis in the hyperbolic spaces can be applied. Our next motivation in this paper is to establish the sharp Hardy--Adams inequality in any dimension. Our next result reads as follows.

\begin{theorem}\label{HARDYADAMS}
Let $m \geq 3$, $n \geq 2m+1$ and $q \geq \frac{2n}{n-1}$. Suppose in addition that $q \leq \frac nm$ if $m$ is odd. Then it holds
\begin{equation}\label{eq:HAineq}
\sup_{u\in W^mL^{\frac nm,q}(\H^n),\, \|\na_g^m u\|_{\frac nm,q}^q -C(n,m,\frac nm)^q \|u\|_{\frac nm,q}^q \leq 1} \int_{\B^n} \exp\big(\beta_{n,m}^{\frac q{q-1}} |u|^{\frac q{q-1}}\big) dx \leqs \infty.
\end{equation}
\end{theorem}
Notice that the condition $m \geq 3$ is crucial in our approach. Indeed, under this condition we can make some estimates for $\|\na_g^m u\|_{\frac nm,q}^q -C(n,m,\frac nm)^q \|u\|_{\frac nm,q}^q$ for which we can apply the results from Theorem \ref{MAINI} and Theorem \ref{MAINII}. We do not know an analogue of \eqref{eq:HAineq} when $m=2$. When $q = \frac nm$, we obtain the following Hardy--Adams inequality
\begin{equation*}
\sup_{u\in W^{m,\frac nm}_0(\H^n),\, \int_{\H^n} |\na_g^m u|^{\frac nm} dV_g -C(n,m,\frac nm)^{\frac nm} \int_{\H^n} |u|^{\frac nm} dV_g \leq 1} \int_{\B^n} \exp\big(\al_{n,m} |u|^{\frac n{n-m}}\big) dx \leqs \infty.
\end{equation*}

The rest of this paper is organized as follows. In Section \S2, we recall some facts on the hyperbolic spaces, the non-increasing rearrangement argument in the hyperbolic spaces and some important results from \cite{Nguyen2020b} which are used in the proof of Theorem \ref{MAINII} and Theorem \ref{HARDYADAMS}. The proof of Theorem \ref{MAINI} is given in Section \S3. Section \S4 is devoted to prove Theorem \ref{MAINII}. Finally, in Section \S5 we provide the proof of Theorem \ref{HARDYADAMS}.

\section{Preliminaries}
We start this section by briefly recalling some basis facts on the hyperbolic spaces and the Lorentz--Sobolev space defined in the hyperbolic spaces. Let $n\geq 2$, a hyperbolic space of dimension $n$ (denoted by $\H^n$) is a complete , simply connected Riemannian manifold having constant sectional curvature $-1$. There are several models for the hyperbolic space $\H^n$ such as the half-space model, the hyperboloid (or Lorentz) model and the Poincar\'e ball model. Notice that all these models are Riemannian isometry. In this paper, we are interested in the Poincar\'e ball model of the hyperbolic space since this model is very useful for questions involving rotational symmetry. In the Poincar\'e ball model, the hyperbolic space $\H^n$ is the open unit ball $B_n\subset \R^n$ equipped with the Riemannian metric
\[
g(x) = \Big(\frac2{1- |x|^2}\Big)^2 dx \otimes dx.
\]
The volume element of $\H^n$ with respect to the metric $g$ is given by
\[
dV_g(x) = \Big(\frac 2{1 -|x|^2}\Big)^n dx,
\]
where $dx$ is the usual Lebesgue measure in $\R^n$. For $x \in B_n$, let $d(0,x)$ denote the geodesic distance between $x$ and the origin, then we have $d(0,x) = \ln (1+|x|)/(1 -|x|)$. For $\rho \geqs 0$, $B(0,\rho)$ denote the geodesic ball with center at origin and radius $\rho$. If we denote by $\na$ and $\Delta$ the Euclidean gradient and Euclidean Laplacian, respectively as well as $\la \cdot, \cdot\ra$ the standard scalar product in $\R^n$, then the hyperbolic gradient $\na_g$ and the Laplace--Beltrami operator $\Delta_g$ in $\H^n$ with respect to metric $g$ are given by
\[
\na_g = \Big(\frac{1 -|x|^2}2\Big)^2 \na,\quad \Delta_g = \Big(\frac{1 -|x|^2}2\Big)^2 \Delta + (n-2) \Big(\frac{1 -|x|^2}2\Big)\la x, \na \ra,
\]
respectively. For a function $u$, we shall denote $\sqrt{g(\na_g u, \na_g u)}$ by $|\na_g u|_g$ for simplifying the notation. Finally, for a radial function $u$ (i.e., the function depends only on $d(0,x)$) we have the following polar coordinate formula
\begin{equation*}
\int_{\H^n} u(x) dx = n \sigma_n \int_0^\infty u(\rho) \sinh^{n-1}(\rho)\,  d\rho.
\end{equation*}

It is now known that the symmetrization argument works well in the setting of the hyperbolic. It is the key tool in the proof of several important inequalities such as the Poincar\'e inequality, the Sobolev inequality, the Moser--Trudinger inequality in $\H^n$. We shall see that this argument is also the key tool to establish the main results in the present paper. Let us recall some facts about the rearrangement argument in the hyperbolic space $\H^n$. A measurable function $u:\H^n \to \R$ is called vanishing at the infinity if for any $t >0$ the set $\{|u| > t\}$ has finite $V_g-$measure, i.e.,
\[
V_g(\{|u|> t\}) = \int_{\{|u|> t\}} dV_g < \infty.
\]
For such a function $u$, its distribution function is defined by
\[
\mu_u(t) = V_g( \{|u|> t\}).
\]
Notice that $t \to \mu_u(t)$ is non-increasing and right-continuous. The non-increasing rearrangement function $u^*$ of $u$ is defined by
\[
u^*(t) = \sup\{s > 0\, :\, \mu_u(s) > t\}.
\] 
The non-increasing, spherical symmetry, rearrangement function $u^\sharp$ of $u$ is defined by
\[
u^\sharp(x) = u^*(V_g(B(0,d(0,x)))),\quad x \in \H^n.
\]
It is well-known that $u$ and $u^\sharp$ have the same non-increasing rearrangement function (which is $u^*$). Finally, the maximal function $u^{**}$ of $u^*$ is defined by
\[
u^{**}(t) = \frac1t \int_0^t u^*(s) ds.
\]
Evidently, $u^*(t) \leq u^{**}(t)$.

For $1\leq p, q < \infty$, the Lorentz space $L^{p,q}(\H^n)$ is defined as the set of all measurable function $u: \H^n \to \R$ satisfying
\[
\|u\|_{L^{p,q}(\H^n)}: = \lt(\int_0^\infty \lt(t^{\frac1p} u^*(t)\rt)^q \frac{dt}t\rt)^{\frac1q} < \infty.
\]
It is clear that $L^{p,p}(\H^n) = L^p(\H^n)$. Moreover, the Lorentz spaces are monotone with respect to second exponent, namely
\[
L^{p,q_1}(\H^n) \subsetneq L^{p,q_2}(\H^n),\quad 1\leq q_1 < q_2 < \infty.
\]
The functional $ u\to \|u\|_{L^{p,q}(\H^n)}$ is not a norm in $L^{p,q}(\H^n)$ except the case $q \leq p$ (see \cite[Chapter $4$, Theorem $4.3$]{Bennett}). In general, it is a quasi-norm which  turns out to be equivalent to the norm obtained replacing $u^*$ by its maximal function $u^{**}$ in the definition of $\|\cdot\|_{L^{p,q}(\H^n)}$. Moreover, as a consequence of Hardy inequality, we have
\begin{proposition}
Given $p\in (1,\infty)$ and $q \in [1,\infty)$. Then for any function $u \in L^{p,q}(\H^n)$ it holds 
\begin{equation}\label{eq:Hardy}
\lt(\int_0^\infty \lt(t^{\frac1p} u^{**}(t)\rt)^q \frac{dt}t\rt)^{\frac1q} \leq \frac p{p-1} \lt(\int_0^\infty \lt(t^{\frac1p} u^*(t)\rt)^q \frac{dt}t\rt)^{\frac1q} = \frac p{p-1} \|u\|_{L^{p,q}(\H^n)}.
\end{equation}
\end{proposition}
For $1\leq p, q \leqs \infty$ and an integer $m\geq 1$, we define the $m-$th order Lorentz--Sobolev space $W^mL^{p,q}(\H^n)$ by taking the completion of $C_0^\infty(\H^n)$ under the quasi-norm
\[
\|\na_g^m u\|_{p,q} := \| |\na_g^m u|\|_{p,q}.
\]
It is obvious that $W^mL^{p,p}(\H^n) = W^{m,p}(\H^n)$ the $m-$th order Sobolev space in $\H^n$. In \cite{Nguyen2020a}, the author established the following P\'olya--Szeg\"o principle in the first order Lorenz--Sobolev spaces $W^1L^{p,q}(\H^n)$ which generalizes the classical P\'olya--Szeg\"o principle in the hyperbolic space.
\begin{theorem}\label{PS}
Let $n\geq 2$, $1\leq q \leq p \leqs \infty$ and $u\in W^{1}L^{p,q}(\H^n)$. Then $u^\sharp \in W^{1}L^{p,q}(\H^n)$ and 
$$\|\na_g u^\sharp\|_{p,q} \leq \|\na_g u\|_{p,q}.$$
\end{theorem}
For $r \geq 0$, define
\[
\Phi(r) = n \int_0^r \sinh^{n-1}(s) ds, \quad r\geq 0,
\]
and let $F$ be the function such that 
\[
r = n \si_n \int_0^{F(r)} \sinh^{n-1}(s) ds, \quad r\geq 0,
\]
i.e., $F(r) = \Phi^{-1}(r/\si_n)$. 

The following results was proved in \cite{Nguyen2020b} (see the Section \S2).
\begin{proposition}
Let $n \geq 2$. Then it holds
\begin{equation}\label{eq:keyyeu}
\sinh^{n}(F(t)) \geqs \frac t{\si_n},\quad t\geqs 0.
\end{equation}
Furthermore, the function
\[
\vphi(t) =\frac{t}{\sinh^{n-1}(F(t))}
\]
is strictly increasing on $(0,\infty)$, and 
\begin{equation}\label{eq:keyyeu*}
\lim_{t\to \infty} \varphi(t) = \frac{n \si_n}{n-1} > \frac{t}{\sinh^{n-1}(F(t))},\quad t >0.
\end{equation}
\end{proposition}
It should be remark that under an extra condition $q \geq \frac{2n}{n-1}$, a stronger estimate which combines both \eqref{eq:keyyeu} and \eqref{eq:keyyeu*} was established by the author in \cite[Lemma $2.1$]{Nguyen2020a} that
\begin{equation*}
\sinh^{q(n-1)}(F(t)) \geq \lt(\frac t{\si_n}\rt)^{q \frac{n-1}n} + \lt(\frac{n-1}n\rt)^q \lt(\frac t{\si_n}\rt)^q,\quad t \geqs 0.
\end{equation*}

Let $u \in C_0^\infty(\H^n)$ and $f = -\Delta_g u$. It was proved by Ngo and the author (see \cite[Proposition $2.2$]{NgoNguyenAMV}) that
\begin{equation}\label{eq:NgoNguyen}
u^*(t) \leq v(t):= \int_t^\infty \frac{s f^{**}(s)}{(n \si_n \sinh^{n-1}(F(s)))^2} ds,\quad t\geqs 0.
\end{equation}

The following results which were proved in \cite{Nguyen2020b,Nguyen2020a} play the important role in the proof of our main results,
\begin{proposition}\label{L1}
Let $p\in (1,n)$ and $\frac{2n}{n-1} \leq q \leq p$. Then we have
\begin{equation}\label{eq:improvedLS1a}
\|\na_g u\|_{p,q}^q - \lt(\frac{n-1}p\rt)^q \|u\|_{p,q}^q \geq \lt(\frac{n-p}p \si_n^{\frac1n}\rt)^q \|u\|_{p^*,q}^q,\quad u\in C_0^\infty(\H^n)
\end{equation}
where $p' = p/(p-1)$,
\end{proposition}

and
\begin{proposition}\label{L2}
Let $n\geq 2$, $p \in (1,n)$ and $q \in (1,\infty)$. 
If $p \in (1,\frac n2)$ then it holds
\begin{equation}\label{eq:LSorder2}
\|\Delta_g u\|_{p,q}^q \geq \lt(\frac{n(n-2p)}{p p'} \si_n^{\frac 2n}\rt)^q \|u\|_{p_2^*,q}^q.
\end{equation}
If $p\in (1,n)$ and $q \geq \frac{2n}{n-1}$ then we have
\begin{equation}\label{eq:improvedLS2}
\|\Delta_g u\|_{p,q}^q - C(n,2,p)^q \|u\|_{p,q}^q \geq \lt(\frac{n^2 \si_n^{\frac2n}}{p'}\rt)^q \int_0^\infty |v'(t)|^q t^{q(\frac1p -\frac2n) + q -1} dt.
\end{equation}
Furthermore, if $p\in (1,\frac n2)$ and $q \geq \frac{2n}{n-1}$ and $\frac{2n}{n-1} \leq q \leq p$ then we have
\begin{equation}\label{eq:improvedLS2a}
\|\Delta_g u\|_{p,q}^q - C(n,2,p)^q \|u\|_{p,q}^q \geq \lt(\frac{n(n-2p)}{p p'} \si_n^{\frac 2n}\rt)^q \|u\|_{p_2^*,q}^q,\quad u \in C_0^\infty(\H^n).
\end{equation}
\end{proposition}
Proposition \ref{L1} follows from \cite[Theorem $1.2$]{Nguyen2020a} while Proposition \ref{L2} follows from Theorem $2.8$ in \cite{Nguyen2020b}.

\section{Proof of Theorem \ref{MAINI}}
In this section, we prove Theorem \ref{MAINI}. The main point is the proof of the case $m=2$. For the case $m\geq 3$, the proof is based on the iteration argument by using the inequalities \eqref{eq:LSorder2} and \eqref{eq:Sob} below.

\begin{proof}[Proof of Theorem \ref{MAINI}] 
We divide the proof of \eqref{eq:AdamsLorentz} into three following cases:\\

\emph{Case 1: $m =2$.} It is enough to consider $u \in C_0^\infty(\H^n)$ with $\|\Delta_g u\|_{\frac n2,q} \leq 1$. Denote $f = -\Delta_g u$ and define $v$ by \eqref{eq:NgoNguyen}, then we have $u^* \leq v$. By \cite[Theorem $1.1$]{Nguyen2020b} , we have $\|u\|_{\frac n2,q}^q \leq C$. Here and in the sequel, we denote by $C$ a generic constant which does not depend on $u$ and whose value maybe changes on each line. For any $t\geqs 0$, we have
\[
\frac n{2q} u^*(t)^q t^{\frac{2q}n} \leq \int_0^t u^*(s)^q s^{\frac{2q}n-1} ds \leq \|u\|_{p,q}^q \leq C,
\]
which yields $u^*(t) \leq C t^{-\frac 2n}$, $t\geqs 0$. Therefore, it is not hard to see that
\[
\Phi_{\frac n2,q}(\beta_{n,2}^{\frac q{q-1}} u^*(t)^{\frac q{q-1}}) \leq C u^*(t)^{\frac{q}{q-1} (j_{\frac n2,q} -1)} \leq C t^{-\frac 2n \frac{q}{q-1} (j_{\frac n2,q} -1)},\quad \forall\, t\geq 1.
\]
By the choice of $j_{\frac n2,q}$, we then have
\begin{equation}\label{eq:tach12}
\int_1^\infty \Phi_{\frac n2,q}(\beta_{n,2}^{\frac q{q-1}} u^*(t)^{\frac q{q-1}}) dt \leq C.
\end{equation}
On the other hand, we have
\begin{align}\label{eq:on01}
\int_0^1 \Phi_{\frac n2,q}(\beta_{n,2}^{\frac q{q-1}} u^*(t)^{\frac q{q-1}}) dt &\leq \int_0^1 \exp\Big(\beta_{n,2}^{\frac q{q-1}} u^*(t)^{\frac q{q-1}}\Big) dt \notag\\
&\leq \int_0^1 \exp\Big(\beta_{n,2}^{\frac q{q-1}} v(t)^{\frac q{q-1}}\Big) dt \notag\\
&= \int_0^\infty \exp\Big(-t + \beta_{n,2}^{\frac q{q-1}} v(e^{-t})^{\frac q{q-1}}) dt.
\end{align}
Notice that
\[
v(e^{-t}) =\int_{e^{-t}}^\infty \frac{r}{(n\si_n \sinh^{n-1}(F(r)))^2} f^{**}(r) dr = \int_{-\infty}^t \frac{e^{-2(1-\frac1n)s}}{(n\si_n \sinh^{n-1}(F(e^{-s})))^2} e^{-\frac{2}ns}f^{**}(e^{-s}) ds.
\]
Denote
\[
\phi(s) = \frac{n-2}{n} e^{-\frac 2n s} f^{**}(e^{-s}),
\]
we then have
\begin{equation}\label{eq:boundnormorder2}
\int_{\R}\phi(s)^q ds = \lt(\frac{n-2}n\rt)^q \int_0^\infty (f^{**}(t) t^{\frac2n})^q \frac{dt}t \leq 1,
\end{equation}
here we used the Hardy inequality \eqref{eq:Hardy} and $\|\Delta_g u\|_{L^{\frac n2,q}(\H^n)} \leq 1$. Define the function
\[
a(s,t) = 
\begin{cases}
\be_{n,2} \frac{n}{n-2} \frac{e^{-2(1-\frac1n)s}}{(n\si_n \sinh^{n-1}(F(e^{-s})))^2} &\mbox{if $s \leq t$,}\\
0&\mbox{if $s > t$.}
\end{cases}
\]
Using the inequality $\si_n \sinh^n(F(r)) \geq r$, we have for $0 \leq s \leq t$
\begin{equation}\label{eq:boundby1order2}
a(s,t) \leq \be_{n,2} \frac1{n(n-2) \si_n^{\frac2n}} = 1.
\end{equation}
Moreover, for $t >0$ we have
\begin{align*}
\int_{-\infty}^0 a(s,t)^{q'} ds + \int_t^\infty a(s,t)^{q'} ds& = \be_{n,2}^{q'} \lt(\frac n{n-2}\rt)^{q'} \int_{-\infty}^0 \lt(\frac{e^{-2(1-\frac1n)s}}{(n\si_n \sinh^{n-1}(F(e^{-s})))^2}\rt)^{q'} ds\\
&\leq \be_{n,2}^{q'} \lt(\frac n{n-2}\rt)^{q'} (n-1)^{-2q'} \int_{-\infty}^0 e^{\frac2n q'} ds\\
&= \be_{n,2}^{q'} \lt(\frac n{n-2}\rt)^{q'} (n-1)^{-2q'} \frac{n}{2q'},
\end{align*}
here we used $n\sigma_n \sinh^{n-1}(F(r)) \geq (n-1) r$. Hence
\begin{equation}\label{eq:dk2Adamsorder2}
\sup_{t >0} \lt(\int_{-\infty}^0 a(s,t)^{q'} ds + \int_t^\infty a(s,t)^{q'} ds\rt)^{\frac1{q'}} \leq \lt(\beta_{n,2}^{q'} \lt(\frac n{n-2}\rt)^{q'} (n-1)^{-2q'} \frac{n}{2q'}\rt)^{\frac1{q'}}.
\end{equation}
Notice that
\begin{equation}\label{eq:majoru*2}
\be_{n,2}v(e^{-t}) \leq \int_{\R} a(s,t) \phi(s) ds.
\end{equation}
With \eqref{eq:on01}, \eqref{eq:boundnormorder2}, \eqref{eq:boundby1order2}, \eqref{eq:dk2Adamsorder2} and \eqref{eq:majoru*2} at hand, we can apply Adams' Lemma \cite{Adams} to obtain
\begin{equation}\label{eq:tporder22}
\int_0^1 \Phi_{\frac n2,q}(\beta_{n,2}^{q'} u^*(t)^{\frac q{q-1}}) dt \leq \int_0^\infty e^{-t + \beta_{n,2}^{q'} v(t)^{q'}} dt \leq C.
\end{equation}
Combining \eqref{eq:tach12} and \eqref{eq:tporder22} together, we arrive
\[
\int_{\R^n} \Phi_{\frac n2,q}(\be_{n,2}^{q'} |u|^{q'}) dx  = \int_0^\infty \Phi_{\frac n2,q}(\be_{n,2}^{q'} (u^*(t))^{q'}) dt \leq C,
\]
for any $u \in W^2L^{\frac n2,q}(\H^n)$ with $\|\Delta_g u\|_{L^{\frac n2,q}(\H^n)} \leq 1$. This proves \eqref{eq:AdamsLorentz} for $m =2$.\\

\emph{Case 2: $m =2k$, $k\geq 2$.} To obtain the result in this case, we apply the iteration argument. Firstly, by iterating the inequality \eqref{eq:LSorder2}, we have that for $k\geq 1$, $q \in (1,\infty)$ and $p \in (1,\frac n{2k})$ 
\[
\|\Delta_g^k u\|_{p,q}^q \geq S(n,2k,p)^q \|u\|_{p_{2k}^*,q}^q.
\]
Hence, if $u \in W^{2k} L^{\frac n{2k},q}(\H^n)$ with $\|\Delta_g^k u\|_{\frac n{2k},q} \leq 1$, then we have 
\[
S(n,2(k-1),\frac n{2k}) \|\Delta_g u\|_{\frac n2,q} \leq 1.
\]
Define $w = S(n,2(k-1),\frac n{2k}) u$, then $\|w\|_{\frac n2,q} \leq 1$. Using the result in the \emph{Case 1} with remark that 
\[
\beta_{n,2k} = \beta_{n,2} S(n,2(k-1),\frac n{2k}),
\]
we obtain
\begin{equation}\label{eq:Case1}
\int_{\H^n} \Phi_{\frac n2,q}(\beta_{n,2k}^{q'} |u|^{q'}) dV_g \leq C.
\end{equation}
By the Lorentz--Poincar\'e inequality \eqref{eq:Poincare}, we have $\|u\|_{\frac n{2k},q}^q \leq C$. Similarly in the \emph{Case 1}, we get $u^*(t) \leq C t^{-\frac{2k}n}$, $t\geqs 0$. Hence, for $t\geq 1$, it holds
\[
\Phi_{\frac n{2k},q}(\beta_{n,2k}^{q'} u^*(t)^{q'}) \leq C (u^*(t))^{q'(j_{\frac n{2k},q} -1)} \leq C t^{-\frac{2k}n q'(j_{\frac n{2k},q} -1)},
\]
which implies
\begin{equation}\label{eq:tach12k}
\int_1^\infty \Phi_{\frac n{2k} ,q}(\beta_{n,2k}^{\frac q{q-1}} u^*(t)^{\frac q{q-1}}) dt \leq C
\end{equation}
by the choice of $j_{\frac n{2k},q}$. Since
\[
\lim_{t\to \infty} \frac{\Phi_{\frac n{2k},q}(t)}{\Phi_{\frac n2,q}(t)} = 1,
\]
then there exists $A$ such that $\Phi_{\frac n{2k},q}(t) \leq 2 \Phi_{\frac n2,q}(t)$ for $t \geq A$. Hence, we have
\begin{align*}
\int_0^1 \Phi_{\frac n{2k},q}(\beta_{n,2k}^{\frac q{q-1}} u^*(t)^{\frac q{q-1}}) dt & = \int_{\{t\in (0,1):u^*(t) \leqs A^{1/q'} \beta_{n,2k}^{-1}\}} \Phi_{\frac n{2k},q}(\beta_{n,2k}^{\frac q{q-1}} u^*(t)^{\frac q{q-1}}) dt\\
&\quad + \int_{\{t\in (0,1):u^*(t) \geq A^{1/q'} \beta_{n,2k}^{-1}\}} \Phi_{\frac n{2k},q}(\beta_{n,2k}^{\frac q{q-1}} u^*(t)^{\frac q{q-1}}) dt\\
&\leq C + 2\int_{\{t\in (0,1):u^*(t) \geq A^{1/q'} \beta_{n,2k}^{-1}\}} \Phi_{\frac n2,q}(\beta_{n,2k}^{\frac q{q-1}} u^*(t)^{\frac q{q-1}}) dt\\
&\leq C + \int_0^1 \Phi_{\frac n2,q}(\beta_{n,2k}^{\frac q{q-1}} u^*(t)^{\frac q{q-1}}) dt\\
&\leq C
\end{align*}
here we have used \eqref{eq:Case1}. Combining the previous inequality together with \eqref{eq:tach12k} proves the result in this case.\\

\emph{Case 3: $m =2k+1$, $k\geq 1$.} Let $f = -\Delta_g^{k} u$. Since $q \leq \frac n{2k+1}$, then it was proved in \cite{Nguyen2020a} (the formula after $(2.8)$ with $u$ replaced by $f$) that 
\[
\|\na_g^m u\|_{\frac n{2k+1},q}^q = \|\na_g f\|_{\frac n{2k+1},q}^q \geq \int_0^\infty |(f^*)'(t)|^q (n \si_n \sinh^{n-1} (F(t)))^q t^{\frac{(2k+1) q}n -1} dt.
\]
Using \eqref{eq:keyyeu}, we have
\[
\|\na_g^m u\|_{\frac n{2k+1},q}^q \geq n^q \si_n^{\frac qn} \int_0^\infty |(f^*)'(t)|^q t^{\frac{2kq}n + q -1} dt.
\]
Applying the one-dimensional Hardy inequality, it holds
\begin{equation}\label{eq:Sob}
\|\na_g^m u\|_{\frac n{2k+1},q}^q \geq (2k)^q \si_n^{\frac qn}\int_0^\infty |f^*(t)|^q t^{\frac{2kq}n -1} dt = (2k)^q \si_n^q \|\Delta_g^k u\|_{\frac n{2k},q}^q.
\end{equation}
For any $u \in W^{2k+1} L^{\frac n{2k+1},q}(\H^n)$ with $\|\na_g^m u\|_{\frac n{2k+1},q} \leq 1$, define $w= 2k \si_n^{\frac1n} u$. By \eqref{eq:Sob}, we have $\|w\|_{\frac n{2k},q}^q \leq 1$. Using the result in  the \emph{Case 2} with remark that 
\[
\beta_{n,2k+1} =2k \si_n^{\frac 1n} \beta_{n,2k},
\]
we obtain
\begin{equation}\label{eq:Case2}
\int_{\H^n} \Phi_{\frac n{2k},q} (\beta_{n,2k+1}^{q'} |u|^{q'}) dV_g \leq C.
\end{equation}
Using \eqref{eq:Case2} together with the last arguments in the proof of the \emph{Case 2} proves the result in this case. \\

It remains to check the sharpness of constant $\beta_{n,m}^{\frac q{q-1}}$. To do this, we construct a sequence of test functions as follows
\[
v_j(x) = \begin{cases}
\frac{(\ln j)^{1/q'}}{\beta_{n,m}} + \frac{n\beta_{n,m}}{2(\ln j)^{1/q}} \sum_{i=1}^{m-1} \frac{(1-j^{\frac 2n}|x|^2)^i}{i} &\mbox{if $0\leq |x| \leq j^{-\frac 1n}$,}\\
-\frac n{\beta_{n,m}} (\ln j)^{-1/q} \ln |x|&\mbox{if $j^{-\frac1n} \leqs |x| \leq 1$,}\\
\xi_j(x) &\mbox{if $1 \leqs |x| \leqs 2$},
\end{cases}
\quad j\geq 2
\]
where $\xi \in C_0^\infty(2\B^n)$ are radial function chosen such that $\xi_j = 0$ on $\pa \B^n$ and for $i=1,\ldots,m-1$
\[
\frac{\pa^i \xi_j}{\pa r^i} \Big{|}_{\pa \B^n} = (-1)^i (i-1)! n\beta_{n,m}^{-1} (\ln j)^{-1/q},
\]
and $\xi_j$, $|\na^l \xi_j|$ and $|\na^m \xi_j|$ are all $O((\ln j)^{-1/q})$ as $j\to \infty$. For $\ep \in (0,1/3)$ let us define $u_{\ep,j}(x) =v_j(x/\ep)$. Then $u_{\ep,j} \in W^m L^{\frac nm,q}(\H^n)$ has support contained in $\{|x| \leq 2\ep\}$. It is easy to check that
\[
|\na_g^m u_{\ep,j}(x)| \leq \lt(\frac{1-|x|^2}2\rt)^m C (\ep^{-1} j^{\frac1n})^m (\ln j)^{-1/q}\leq C2^{-m}(\ep^{-1} j^{\frac1n})^m (\ln j)^{-1/q}
\]
for $|x| \leq \ep j^{-\frac1n}$, and
\[
|\na_g^m u_{\ep,j}(x)| \leq C \ep^{-m} (\ln j)^{-\frac1q}\lt(\frac{1-|x|^2}2\rt)^m \leq C2^{-m}\ep^{-m} (\ln j)^{-\frac1q}
\]
for $|x|\in (\ep, 2\ep)$ with a positive constant $C$ independent of $\ep\leqs \frac13$ and $j$. Furthermore, we can check that
\begin{align*}
|\na^m_g u_{\ep,j}(x)|&\leq \lt(\frac{1-|x|^2}2\rt)^m\lt((|x|^n \si_n)^{-\frac mn} + C |x|^{-m+1}\rt) (\ln j)^{-\frac1q}\\
& \leq 2^{-m}(\ln j)^{-\frac1q} \lt((|x|^n \si_n)^{-\frac mn} + C |x|^{-m+1}\rt)
\end{align*}
and
\begin{align*}
|\na^m_g u_{\ep,j}(x)|&\geq \lt(\frac{1-|x|^2}2\rt)^m\lt((|x|^n \si_n)^{-\frac mn} - C |x|^{-m+1}\rt) (\ln j)^{-\frac1q}\\
& \geq \lt(\frac{1-\ep^2}2\rt)^{-m}(\ln j)^{-\frac1q} \lt((|x|^n \si_n)^{-\frac mn} - C |x|^{-m+1}\rt)
\end{align*}
for $|x| \in (\ep j^{-\frac1n}, \ep )$ with $\ep \geqs 0$ small enough where $C$ is a positive constant independent of $\ep$ and $j$. Define
\[
h_1(x) = \begin{cases}
C2^{-m}(\ep^{-1} j^{\frac1n})^m (\ln j)^{-1/q}&\mbox{if $|x| \leq \ep j^{-\frac1n}$}\\
2^{-m}(\ln j)^{-\frac1q} \lt((|x|^n \si_n)^{-\frac mn} + C |x|^{-m+1}\rt)&\mbox{if $|x| \in (\ep j^{-\frac1n}, \ep )$}\\
C2^{-m}\ep^{-m} (\ln j)^{-\frac1q}&\mbox{if $|x|\in (\ep, 2\ep)$}\\
0&\mbox{if $|x| \in (2\ep,1)$},
\end{cases}
\]
Then we have $0 \leq |\na_g^m u| \leq h_1$. Consequently, we get $0 \leq |\na^m_g u|^* \leq h_1^*$. Let us denote by $h_1^{*,e}$  the rearrangement function of $h_1$ with respect to Lebesgue measure. Since the support of $h_1$ is contained in $\ep \{|x| \leq \ep\}$, then we can easy check that
\[
h_1^*(t) \leq h_1^{*,e}\lt(\Big(\frac{1-\ep^2}2\Big)^n t\rt).
\]
Consequently, we have
\[
\|\na_g^m u_{\ep,j}\|_{\frac nm,q}^q \leq \lt(\frac2{1-\ep^2}\rt)^{mq} \int_0^\infty h_1^{*,e}(t)^q t^{\frac{mq}n -1} dt
\]
Notice that by enlarging the constant $C$ (which is still independent of $\ep$ and $j$), we can assume that 
\[
C2^{-m}\ep^{-m} (\ln j)^{-\frac1q} \geq h_1\Big |_{\{|x| = \ep\}} =2^{-m}(\ln j)^{-\frac1q} \ep^{-m}\lt(\si_n^{-\frac mn} + C \ep \rt)
\]
for $\ep \geqs 0$ small enough. For $j$ larger enough, we can chose $x_0$ with $\ep j^{-\frac1n} \leqs |x_0| \leq \ep$ such that $C2^{-m}\ep^{-m} (\ln j)^{-\frac1q} = h_1(x_0)$. It is easy to see that $c\ep \leq |x_0| \leq C \ep$ for constant $C, c \geqs 0$ independent of $\ep$ and $j$. We have
\[
h_1(x) \leq g(x) :=\begin{cases}
h_1(x)&\mbox{if $|x| \leq |x_0| $}\\
C2^{-m}\ep^{-m} (\ln j)^{-\frac1q}&\mbox{if $|x|\in (|x_0|, 2\ep)$}\\
0&\mbox{if $|x| \geq 2\ep$}.
\end{cases}
\]
Notice that $g$ is non-increasing radially symmetric function in $\B^n$, hence $g^{\sharp,e} = g$. Using the function $g$, we can prove that
\[
\int_0^\infty h_1^{*,e}(t)^q t^{\frac{mq}n -1} dt \leq 2^{-mq}(1 + C (\ln j)^{-1}).
\]
Therefore, we have
\[
\|\na_g^m u_{\ep,j}\|_{\frac nm,q}^q \leq \lt(\frac1{1-\ep^2}\rt)^{mq}(1 + C (\ln j)^{-1})
\]
Set $w_{\ep,j} = u_{\ep,j}/\|\na_g^m u_{\ep,j}\|_{\frac nm,q}$. For any $\beta \geqs \beta_{n,m}^{q'}$, we choose $\ep \geqs 0$ small enough such that $\gamma := \beta (1-\ep^2)^{\frac{mq}{q-1}}\geqs \be_{n,m}^{q'}$. Then we have
\begin{align*}
\int_{\H^n} \Phi_{\frac nm,q}(\beta |w_{\ep,j}|^{q'}) dV_g &\geq \int_{\{|x|\leq \ep j^{-\frac1n}\}} \Phi_{\frac nm,q}\Big(\frac{\beta}{\|\na_g^m u_{\ep,j}\|_{\frac nm,q}^{q'}} |u_{\ep,j}|^{q'}\Big) dV_g\\
&\int_{\{|x|\leq \ep j^{-\frac1n}\}} \Phi_{\frac nm,q}\Big(\frac{\ga}{(1+ C (\ln j)^{-1})^{q'}} |u_{\ep,j}|^{q'}\Big) dV_g\\
&\geq 2^n \int_{\{|x|\leq \ep j^{-\frac1n}\}} \Phi_{\frac nm,q}\Big(\frac{\ga}{(1+ C (\ln j)^{-1})^{q'}} |u_{\ep,j}|^{q'}\Big) dx\\
&= 2^n \ep^{n}\int_{\{|x|\leq  j^{-\frac1n}\}} \Phi_{\frac nm,q}\Big(\frac{\ga}{(1+ C (\ln j)^{-1})^{q'}} |v_{j}|^{q'}\Big) dx\\
&\geq 2^n \ep^{n}\int_{\{|x|\leq  j^{-\frac1n}\}} \Phi_{\frac nm,q}\Big(\frac{\ga}{\beta_{n,m}^{q'}}\frac{\ln j}{(1+ C (\ln j)^{-1})^{q'}} \Big) dx\\
&=2^n \ep^{n}\si_n \Phi_{\frac nm,q}\Big(\frac{\ga}{\beta_{n,m}^{q'}}\frac{\ln j}{(1+ C (\ln j)^{-1})^{q'}} \Big) e^{-\ln j}.
\end{align*}
Since 
\[
\lim_{j\to \infty} \frac{\ga}{\beta_{n,m}^{q'}}\frac{\ln j}{(1+ C (\ln j)^{-1})^{q'}} = \infty,
\]
then
\[
\Phi_{\frac nm,q}\Big(\frac{\ga}{\beta_{n,m}^{q'}}\frac{\ln j}{(1+ C (\ln j)^{-1})^{q'}} \Big) \geq C e^{\frac{\ga}{\beta_{n,m}^{q'}}\frac{\ln j}{(1+ C (\ln j)^{-1})^{q'}}}
\]
for $j$ larger enough. Consequently, we get
\[
\int_{\H^n} \Phi_{\frac nm,q}(\beta |w_{\ep,j}|^{q'}) dV_g \geq 2^n \ep^n \si_n C e^{\frac{\ga}{\beta_{n,m}^{q'}}\frac{\ln j}{(1+ C (\ln j)^{-1})^{q'}} -\ln j} \to \infty
\]
as $j\to \infty$ since $\ga \geqs \beta_{n,m}^{q'}$. This proves the sharpness of $\beta_{n,m}^{q'}$.

The proof of Theorem \ref{MAINI} is then completely finished. 

\end{proof}

\section{Proof of Theorem \ref{MAINII}}

This section is devoted to prove Theorem \ref{MAINII}. The proof is based on the inequalities \eqref{eq:improvedLS1a} and \eqref{eq:improvedLS2a}, the iteration argument and Theorem \ref{MAINII} for $m\geq 3$. The case $m=2$ is proved by using inequality \eqref{eq:improvedLS2} and the Moser--Trudinger inequality involving to the fractional dimension in Lemma \ref{MT} below. Let $\theta \geqs 1$, we denote by $\lam_\theta$ the measure on $[0,\infty)$ of density
\[
d\lam_{\theta} = \theta \sigma_{\theta} x^{\theta -1} dx, \quad \sigma_{\theta} = \frac{ \pi^{\frac\theta 2}}{\Gamma(\frac\theta 2+ 1)}.
\]
For $0\leqs R \leq \infty$ and $1\leq p \leqs \infty$, we denote by $L_\theta^p(0,R)$  the weighted Lebesgue space of all measurable functions $u: (0,R) \to \R$ for which
\[
\|u\|_{L^p_\theta(0,R)}= \lt(\int_0^R |u|^p d\lam_\theta\rt)^{\frac1p} \leqs \infty.
\]
Besides, we define
\[
W^{1,p}_{\al,\theta}(0,R) =\Big\{u\in L^p_\theta(0,R)\, :\, u' \in L_\alpha^p(0,R),\,\, \lim_{x\to R^{-}} u(x) =0\Big\}, \quad \al, \theta \geqs 1.
\]
In \cite{deOliveira}, de Oliveira and do \'O prove the following sharp Moser--Trudinger inequality involving the measure $\lam_\theta$: suppose $0 \leqs R \leqs \infty$ and $\alpha \geq 2, \theta \geq 1$, then 
\begin{equation}\label{eq:MTOO}
D_{\al,\theta}(R) :=\sup_{u\in W^{1,\al}_{\al,\theta}(0,R),\, \|u'\|_{L^\alpha_\alpha(0,R)} \leq 1} \int_0^R e^{\mu_{\al,\theta} |u|^{\frac{\alpha}{\alpha -1}}} d\lam_\theta \leqs \infty
\end{equation}
where $\mu_{\al,\theta} = \theta \alpha^{\frac1{\al -1}} \sigma_{\al}^{\frac1{\alpha -1}}$. Denote $D_{\al,\theta} = D_{\al,\theta}(1)$. It is easy to see that $D_{\al,\theta}(R) = D_{\al,\theta} R^\theta$. 
\begin{lemma}\label{MT}
Let $\alpha \geqs 1$ and $q\geq 2$. There exists a constant $C_{\al,q} \geqs 0$ such that for any $u \in W^{1,q}_{q,\al}(0,\infty),$ $u' \leq 0$ and $\|u\|_{L^q_{\al}(0,\infty)}^q + \|u'\|_{L^q_q(0,\infty)}^q\leq 1$, it holds
\begin{equation}\label{eq:Abreu}
\int_0^\infty \Phi_{\frac q\al,q}(\mu_{q,1} |u|^{\frac q{q -1}}) d\lam_1 \leq C_{\al,q}.
\end{equation}
\end{lemma}
\begin{proof}
We follows the argument in \cite{Ruf2005}. Since $u' \leq 0$ then $u$ is a non-increasing function. Hence, for any $t \geqs 0$, it holds
\begin{equation}\label{eq:boundu}
u(r)^q \leq \frac{1}{\si_\al r^\al} \int_0^r u(s)^q d\lam_\al \leq \frac{\int_0^\infty u(s)^q d\lam_\al}{\si_\al r^\al} \leq \frac{\|u\|_{L^q_\al(0,\infty)}^q}{\si_\al r^\al}.
\end{equation}
For $R \geqs 0$, define $w(r) = u(r) - u(R)$ for $r \leq R$ and $w(r) =0$ for $r \geqs R$. Then $w \in W^{1,q}{q,q}(0,R)$ and 
\begin{equation}\label{eq:on0R}
\|w\|_{L^q_q(0,R)}^q = \int_0^R |u'(s)|^q d\lam_q \leq 1 - \|u\|_{L^q_\al(0,\infty)}^q.
\end{equation}
For $r \leq R$, we have $u(r) = w(r) + u(R)$. Since $q \geq 2$, then there exists $C \geqs 0$ depending only on $q$ such that
\[
u(r)^{q'} \leq w(r)^{q'} + C w(r)^{q'-1} u(R) + u(R)^{q'}.
\]
Applying Young's inequality and \eqref{eq:boundu}, we get
\begin{align}\label{eq:E1}
u(r)^{q'} &\leq w(r)^{q'}\lt(1+ \frac{C}{q} u(R)^q\rt) + \frac{q-1}q + u(R)^{q'}\notag\\
&\leq w(r)^{q'}\lt(1+ \frac{C}{q\si_\al R^\al}\rt) + \frac{q-1}q + \lt(\frac1{\si_\al R^\al}\rt)^{q'-1}.
\end{align}
Fix a $R \geq 1$ large enough such that $\frac{C}{q\si_\al R^\al} \leq 1$, and set
\[
v(r) = w(r) \lt(1+ \frac{C}{q\si_\al R^\al}\rt)^{\frac{q-1}q}.
\]
Using \eqref{eq:on0R} and the choice of $R$, we can easily verify that $\|v\|_{L^q_q(0,R)}^q \leq 1$. Hence, applying \eqref{eq:MTOO}, we get
\begin{equation}\label{eq:on0R1}
\int_0^R e^{\mu_{q,1} |u|^{q'}} d\lam_1 \leq D_{q,1} R.
\end{equation}
For $r \geq R$, we have $u(r) \leq \si_\al^{-\frac1q} R^{-\frac\al q}$, hence it holds
\[
\Phi_{\frac q\al,q} (\mu_{q,1}  |u(r)|^{q'}) \leq C |u(r)|^{q'(j_{\al,q}-1)} \leq C r^{-\frac{\al}{q-1}(j_{\al,q} -1)}.
\]
By the choice of $j_{\al,q}$, we have
\begin{equation}\label{eq:E2}
\int_R^\infty \Phi_{\frac q\al,q}(\mu_{q,1}  |u(r)|^{q'}) d\lam_1 \leq C.
\end{equation}
Putting \eqref{eq:E1}, \eqref{eq:on0R1}, \eqref{eq:E2} together and using $R\geq 1$, we get
\begin{align*}
\int_0^\infty \Phi_{\frac q\al,q}(\mu_{q,1} |u|^{q'}) d\lam_1 &\leq \int_0^R \Phi_{\frac q\al,q}(\mu_{q,1}  |u|^{q'}) d\lam_1 + \int_R^\infty \Phi_{\frac q\al,q}(\mu_{q,1}  |u|^{q'}) d\lam_1 \\
&\leq \int_0^R \exp\Big(\mu_{q,1}  |u|^{q'}\Big) d\lam_1 + C\\
&\leq \int_0^R \exp\Big(\mu_{q,1}  v^{q'} + \mu_{q,1} \big(\frac{q-1}q + \si_\al^{-\frac1{q-1}}\big)\Big) d\lam_1 + C\\
&\leq \exp\Big(\mu_{q,1} \big(\frac{q-1}q + \si_\al^{-\frac1{q-1}}\big)\Big)D_{q,1} R + C\\
&\leq C.
\end{align*}
\end{proof}

For any $\tau \geqs 0$ and $u \in W^{1,q}_{q,\al}(0,\infty),$ such that $u' \leq 0$ and $\tau \|u\|_{L^q_\alpha(0,\infty)}^q + \|u'\|_{L^q_q(0,\infty)}^q\leq 1$. Applying \eqref{eq:Abreu} for function $u_\tau(x) = u(\tau^{-\frac1\al} x)$ and making the change of variables, we obtain
\begin{equation}\label{eq:Abreu1}
 \int_0^\infty \Phi_{\frac q \alpha,q}(\mu_{q,1}  |u|^{q'}) d\lam_1 \leq C \tau^{-\frac 1\al}.
\end{equation}
 
We are now ready to give the proof of Theorem \ref{MAINII}.

\begin{proof}[Proof of Theorem \ref{MAINII}]
We divide the proof into the following cases.\\

\emph{Case 1: $m=2$.} Let $u \in C_0^\infty(\H^n)$ with $\|\Delta_g u\|_{\frac n2,q}^q - \lam \|u\|_{\frac n2,q}^q \leq 1$. Define $v$ by \eqref{eq:NgoNguyen} and $\tilde v(x) = v(V_g(B(0,d(0,x))))$, then $u^* \leq v$, $\|\Delta_g u\|_{\frac n2,q} = \|\Delta_g \tilde v\|_{\frac n2,q}$ and $\|u\|_{\frac n2,q} \leq \|\tilde v\|_{\frac n2,q}$. So, we have
\[
\|\Delta_g \tilde v\|_{\frac n2,q}^q -\lam \|\tilde v\|_{\frac n2,q}^q \leq 1.
\]
We show that $\int_{\H^n} \Phi_{\frac n2,q}(\beta_{n,2} |\tilde v|^{q'}) dV_g \leq C.$ Set $\kappa = C(n,2,n/2)^q -\lam \geqs 0$. Applying the inequality \eqref{eq:improvedLS2} for $\tilde v$, we get
\[
\lt(n(n-2) \si_n^{\frac2n}\rt)^q \int_0^\infty |v'(t)|^q t^{ q -1} dt + \kappa \int_0^\infty v(t)^q t^{\frac{2q}n -1} dt \leq 1.
\]
Define
\[
w = \frac{n(n-2) \si_n^{\frac2n}}{(q \si_q)^{\frac1q}} v,\quad \tau = \frac{q \si_q}{(n(n-2) \si_n^{\frac2n})^q \frac{2q}n \si_{\frac{2q}n}}\kappa,
\]
then, we have
\[
\int_0^\infty |w'|^q d\lam_q + \tau \int_0^\infty |w|^q d \lam_{\frac{2q}n} \leq 1.
\]
Applying the inequality \eqref{eq:Abreu1}, we obtain
\[
\int_0^\infty \Phi_{\frac{n}2,q} (\mu_{q,1} w^{\frac q{q-1}}) d\lam_1 \leq C_{\frac{2q}n,q} \tau^{-\frac n{2q}}.
\]
Notice that
\[
\int_{\H^n} \Phi_{\frac n2,q} (\beta_{n,2}^{q'} |\tilde v|^{q'}) dV_g = \frac12 \int_0^\infty \Phi_{\frac{n}2,q}(\beta_{n,2}^{q'} |v|^{q'}) d\lam_1=\frac12 \int_0^\infty \Phi_{\frac{n}2,q} (\mu_{q,1} w^{\frac q{q-1}}) d\lam_1.
\]
Hence, it holds
\[
\int_{\H^n} \Phi_{\frac n2,q} (\beta_{n,m}^{q'} |\tilde v|^{q'}) dV_g \leq \frac12 C_{\frac{2q}n,q} \tau^{-\frac n{2q}}.
\]
This completes the proof of this case.\\

\emph{Case 2: $m = 2k$, $k\geq 2$.} Denote $\tau = C(n,2k, \frac n{2k})^q -\lam \geqs 0$. We have
\[
1\geq \|\Delta^k_g u\|_{\frac n{2k},q}^q - \lam \|u\|_{\frac n{2k},q}^q \geq \tau \|u\|_{\frac n{2k},q}^q,
\]
which yields
\begin{equation}\label{eq:normu}
\|u\|_{\frac n{2k},q}^q \leq \tau^{-1}.
\end{equation}
On the other hand, by the Lorentz--Poincar\'e inequality \eqref{eq:Poincare} and the Poincar\'e--Sobolev inequality under Lorentz--Sobolev norm \eqref{eq:improvedLS2a}, we have
\begin{align*}
\|\Delta^k_g u\|_{\frac n{2k},q}^q - \lam \|u\|_{\frac n{2k},q}^q &\geq \|\Delta^k_g u\|_{\frac n{2k},q}^q - C(n,2k,\frac{n}{2k})\|u\|_{\frac n{2k},q}^q + \tau \|u\|_{\frac n{2k},q}^q\\
&\geq \|\Delta^k_g u\|_{\frac n{2k},q}^q - C(n,2,\frac{n}{2k})\|\Delta^{k-1}_g u\|_{\frac n{2k},q}^q + \tau \|u\|_{\frac n{2k},q}^q\\
&\geq (2(k-1)(n-2k) \si_n^{\frac 2n})^q \|\Delta^{k-1}_g u\|_{\frac n{2(k-1)},q}^ q + \tau \|u\|_{\frac n{2k},q}^q.
\end{align*}
Set
$
w = 2(k-1)(n-2k) \si_n^{\frac 2n} u
$
we have $\|\Delta^{k-1}_g w\|_{\frac n{2(k-1)},q}^ q \leq 1$. Applying the Adams inequality \eqref{eq:AdamsLorentz}, we obtain
\[
\int_{\H^n} \Phi_{n,2(k-1),q}(\beta_{n,2k}^{q'} |u|^{q'}) dV_g = \int_{\H^n} \Phi_{n,2(k-1),q}(\beta_{n,2(k-1)}^{q'} |w|^{q'}) \leq C,
\]
here we use
\[
\beta_{n,2k} = 2(k-1)(n-2k) \si_n^{\frac 2n} \beta_{n,2(k-1)}.
\]
Using \eqref{eq:normu} and repeating the last argument in the proof of \emph{Case 2} in the proof of Theorem \ref{MAINI}, we obtain \eqref{eq:improvedAL} in this case.\\

\emph{Case 3: $m =2k+1$, $k\geq 1$.} Denote $\tau = C(n,2k+1,\frac n{2k+1})^q -\tau \geqs 0$. Since $\frac{2n}{n-1} \leq q \leq \frac n{2k+1}$, then using the Lorentz--Poincar\'e inequality \eqref{eq:Poincare} and the Poincar\'e--Sobolev inequality under Lorentz--Sobolev norm \eqref{eq:improvedLS1a}, we get
\begin{align*}
1&\geq \|\na_g \Delta_g^k u\|_{\frac n{2k+1},q}^q - \lam \|u\|_{\frac n{2k+1},q}^q \\
&\geq \|\na_g \Delta_g^k u\|_{\frac n{2k+1},q}^q - C(n,2k+1,\frac n{2k+1})^q \|u\|_{\frac n{2k+1},q}^q + \tau \|u\|_{\frac n{2k+1},q}^q\\
&\geq \|\na_g \Delta_g^k u\|_{\frac n{2k+1},q}^q - \lt(\frac{(2k+1)(n-1)}{n}\rt)^q \|\Delta_g^k u\|_{\frac n{2k+1},q}^q + \tau \|u\|_{\frac n{2k+1},q}^q\\
&\geq (2k \si_n^{\frac 1n})^q \|\Delta_g^k u\|_{\frac n{2k},q}^q + \tau \|u\|_{\frac n{2k+1},q}^q.
\end{align*}
We now can use the argument in the proof of \emph{Case 2} to obtain the result in this case. The proof of Theorem \ref{MAINI} is then completely finished.
\end{proof}

\section{Proof of Theorem \ref{HARDYADAMS}}
In this section, we provide the proof of Theorem \ref{HARDYADAMS}. The proof uses the Lorentz--Poincar\'e inequality \eqref{eq:Poincare}, the Poincar\'e--Sobolev inequality under Lorentz--Sobolev norm \eqref{eq:improvedLS1a} and \eqref{eq:improvedLS2a}, and the Adams type inequality \eqref{eq:AdamsLorentz}.

\begin{proof}[Proof of Theorem \ref{HARDYADAMS}]
We divide the proof in two cases according to the facts that $m$ is even or odd.\\

\emph{Case 1: $m=2k$, $k\geq 2$.} Using the Lorentz--Poincar\'e inequality \eqref{eq:Poincare} and the inequality \eqref{eq:improvedLS2a}, we have
\begin{align*}
1\geq \|\Delta^k_g u\|_{\frac n{2k},q}^q - C(n,2k,\frac{n}{2k})\|u\|_{\frac n{2k},q}^q &\geq \|\Delta^k_g u\|_{\frac n{2k},q}^q - C(n,2,\frac{n}{2k})\|\Delta^{k-1}_g u\|_{\frac n{2k},q}^q\notag\\
&\geq (2(k-1)(n-2k) \si_n^{\frac 2n})^q \|\Delta^{k-1}_g u\|_{\frac n{2(k-1)},q}^q.
\end{align*}
Let us define the function $w$ by $w = 2(k-1)(n-2k) \si_n^{\frac 2n} u$. Then we have $\|\Delta^{k-1}_g w\|_{\frac n{2(k-1)},q}^ q \leq 1$. Applying the Adams type inequality \eqref{eq:AdamsLorentz}, we obtain
\begin{equation}\label{eq:aa0}
\int_{\H^n} \Phi_{n,2(k-1),q}(\beta_{n,2k}^{q'} |u|^{q'}) dV_g = \int_{\H^n} \Phi_{n,2(k-1),q}(\beta_{n,2(k-1)}^{q'} |w|^{q'}) dV_g\leq C,
\end{equation}
here we use
\[
\beta_{n,2k} = 2(k-1)(n-2k) \si_n^{\frac 2n} \beta_{n,2(k-1)}.
\]
It follows from \eqref{eq:aa0} and the fact $\Phi_{n,2(k-1),q}(t) \geq C t^{j_{\frac{n}{2(k-1)},q}-1}$ that 
\[
\int_0^\infty (u^*(t))^{q'(j_{\frac{n}{2(k-1)},q}-1)} dt = \int_{\H^n} |u|^{q'(j_{\frac{n}{2(k-1)},q}-1)} dV_g \leq C.
\]
Using the non-increasing of $u^*$, we can easily verify that 
\[
u^*(t) \leq C t^{-1/(q'(j_{\frac{n}{2(k-1)},q}-1))}
\]
for any $t \geqs 0$. Let $x_0 \in \B^n$ such that $V_g(B(0,d(0,x_0))) = 1$. Since the function $h(x) = (1-|x|^2)^n$ is decreasing with respect to $d(0,|x|)$, then $h^\sharp = h$. Using Hardy--Littlewood inequality, we have
\begin{align}\label{eq:aa1}
\int_{\B^n} e^{\beta_{n,2k}^{q'} |u|^{q'}} dx = 2^{-n} \int_{\H^n} e^{\beta_{n,2k}^{q'} |u|^{q'}} h(x) dV_g& \leq 2^{-n} \int_{\H^n} e^{\beta_{n,2k}^{q'} |u^\sharp|^{q'}} h(x) dV_g\notag\\
&= 2^{-n} \int_0^\infty e^{\beta_{n,2k}^{q'} |u^*(t)|^{q'}} h(t) dt.
\end{align}
For $t \geq 1$ we have $u^*(t) \leq C$, hence it holds
\begin{equation}\label{eq:aa2}
2^{-n}\int_1^\infty e^{\beta_{n,2k}^{q'} |u^*(t)|^{q'}} h(t) dt \leq C2^{-n} \int_1^\infty h(t) dt =C \int_{\{|x| \geq |x_0|\}} dx \leq C\si_n.
\end{equation}
Notice that
\[
e^t = \Phi_{\frac n{2(k-1)},q}(t) + \sum_{j=0}^{j_{\frac{n}{2(k-1)},q} -2} \frac{t^j}{j!}.
\]
Using Young's inequality, we get
\[
e^t = \Phi_{\frac n{2(k-1)},q}(t) + C(1+ t^{j_{\frac{n}{2(k-1)},q} -2}).
\]
Consequently, by using the previous inequality and the inequality \eqref{eq:aa0} and the fact $h\leq 1$, we obtain
\begin{align}\label{eq:aa3}
\int_0^1 e^{\beta_{n,2k}^{q'} |u^*(t)|^{q'}} h(t) dt&\leq \int_0^1 \Phi_{\frac n{2(k-1)},q}(\beta_{n,2k}^{q'} |u^*(t)|^{q'}) dt + C \int_0^1\lt(1 + (u^*(t))^{q'(j_{\frac{n}{2(k-1)},q} -2)}\rt) dt\notag\\
&\leq \int_0^\infty \Phi_{\frac n{2(k-1)},q}(\beta_{n,2k}^{q'} |u^*(t)|^{q'}) dt + C + C\int_0^1(u^*(t))^{q'(j_{\frac{n}{2(k-1)},q} -2)} dt\notag\\
&\leq \int_{\H^n} \Phi_{n,2(k-1),q}(\beta_{n,2k}^{q'} |u|^{q'}) dV_g + C + C \int_0^1 t^{-\frac{j_{\frac{n}{2(k-1)},q} -2}{j_{\frac{n}{2(k-1)},q} -1}} dt\notag\\
&\leq C.
\end{align}
Combining \eqref{eq:aa1}, \eqref{eq:aa2} and \eqref{eq:aa3} we obtain the desired estimate.\\

\emph{Case 2: $m=2k+1$, $k\geq 1$.} Since $\frac{2n}{n-1} \leq q \leq \frac{n}{2k+1}$, then by using the Lorentz--Poincar\'e inequality \eqref{eq:Poincare} and the Poincar\'e--Sobolev inequality under Lorentz--Sobolev norm \eqref{eq:improvedLS1a}, we get
\begin{align*}
1&\geq \|\na_g \Delta_g^k u\|_{\frac n{2k+1},q}^q - C(n,2k+1,\frac n{2k+1})^q \|u\|_{\frac n{2k+1},q}^q \\
&\geq \|\na_g \Delta_g^k u\|_{\frac n{2k+1},q}^q - \lt(\frac{(2k+1)(n-1)}{n}\rt)^q \|\Delta_g^k u\|_{\frac n{2k+1},q}^q\\
&\geq (2k \si_n^{\frac 1n})^q \|\Delta_g^k u\|_{\frac n{2k},q}^q.
\end{align*}
Setting $w = 2k \si_n^{\frac 1n} u$, we have $\|\Delta^k_g w\|_{\frac n{2k}, q}^q \leq 1$. Applying the Adams type inequality \eqref{eq:AdamsLorentz}, we obtain
\begin{equation}\label{eq:bb0}
\int_{\H^n} \Phi_{n,2k,q}(\beta_{n,2k+1}^{q'} |u|^{q'}) dV_g = \int_{\H^n} \Phi_{n,2k,q}(\beta_{n,2k}^{q'} |w|^{q'}) dV_g\leq C,
\end{equation}
here we use
\[
\beta_{n,2k+1} = 2k \si_n^{\frac 1n} \beta_{n,2k}.
\]
Similarly in the \emph{Case 1}, the inequality \eqref{eq:bb0} yields
\[
\int_0^\infty (u^*(t))^{q'(j_{\frac{n}{2k},q}-1)} dt = \int_{\H^n} |u|^{q'(j_{\frac{n}{2k},q}-1)} dV_g \leq C,
\]
which implies
\[
u^*(t) \leq C t^{-\frac1{q'(j_{\frac{n}{2k},q}-1)}},\quad t \geqs 0.
\]
Repeating the last arguments in the proof of \emph{Case 1}, we obtain the result in this case.\\

The proof of Theorem \ref{HARDYADAMS} is then completed.

\end{proof}


\bibliographystyle{abbrv}

\begin{thebibliography}{10}

\bibitem{Adachi00}
S.~Adachi and K.~Tanaka.
\newblock Trudinger type inequalities in {$\bold R^N$} and their best
  exponents.
\newblock {\em Proc. Amer. Math. Soc.}, 128(7):2051--2057, 2000.

\bibitem{Adams}
D.~R. Adams.
\newblock A sharp inequality of {J}. {M}oser for higher order derivatives.
\newblock {\em Ann. of Math. (2)}, 128(2):385--398, 1988.

\bibitem{AdimurthiDruet2004}
Adimurthi and O.~Druet.
\newblock Blow-up analysis in dimension 2 and a sharp form of
  {T}rudinger-{M}oser inequality.
\newblock {\em Comm. Partial Differential Equations}, 29(1-2):295--322, 2004.

\bibitem{AdimurthiSandeep2007}
Adimurthi and K.~Sandeep.
\newblock A singular {M}oser-{T}rudinger embedding and its applications.
\newblock {\em NoDEA Nonlinear Differential Equations Appl.}, 13(5-6):585--603,
  2007.

\bibitem{AdimurthiTintarev2010}
Adimurthi and K.~Tintarev.
\newblock On a version of {T}rudinger-{M}oser inequality with {M}\"{o}bius
  shift invariance.
\newblock {\em Calc. Var. Partial Differential Equations}, 39(1-2):203--212,
  2010.

\bibitem{AdimurthiYang2010}
Adimurthi and Y.~Yang.
\newblock An interpolation of {H}ardy inequality and {T}rundinger-{M}oser
  inequality in {$\Bbb R^N$} and its applications.
\newblock {\em Int. Math. Res. Not. IMRN}, (13):2394--2426, 2010.

\bibitem{Alberico}
A.~Alberico.
\newblock Moser type inequalities for higher-order derivatives in {L}orentz
  spaces.
\newblock {\em Potential Anal.}, 28(4):389--400, 2008.

\bibitem{Alvino1996}
A.~Alvino, V.~Ferone, and G.~Trombetti.
\newblock Moser-type inequalities in {L}orentz spaces.
\newblock {\em Potential Anal.}, 5(3):273--299, 1996.

\bibitem{Balogh}
Z.~M. Balogh, J.~J. Manfredi, and J.~T. Tyson.
\newblock Fundamental solution for the {$Q$}-{L}aplacian and sharp
  {M}oser-{T}rudinger inequality in {C}arnot groups.
\newblock {\em J. Funct. Anal.}, 204(1):35--49, 2003.

\bibitem{Bennett}
C.~Bennett and R.~Sharpley.
\newblock {\em Interpolation of operators}, volume 129 of {\em Pure and Applied
  Mathematics}.
\newblock Academic Press, Inc., Boston, MA, 1988.

\bibitem{Bertrand}
J.~Bertrand and K.~Sandeep.
\newblock Adams inequality on pinched hadamard manifolds.
\newblock {\em preprint, arXiv:1809.00879}, 2019.

\bibitem{Carleson86}
L.~Carleson and S.-Y.~A. Chang.
\newblock On the existence of an extremal function for an inequality of {J}.
  {M}oser.
\newblock {\em Bull. Sci. Math. (2)}, 110(2):113--127, 1986.

\bibitem{CassaniTarsi2009}
D.~Cassani and C.~Tarsi.
\newblock A {M}oser-type inequality in {L}orentz-{S}obolev spaces for unbounded
  domains in {$\Bbb R^N$}.
\newblock {\em Asymptot. Anal.}, 64(1-2):29--51, 2009.

\bibitem{Chen}
L.~Chen, G.~Lu, and M.~Zhu.
\newblock Existence and nonexistence of extremals for critical adams
  inequalities in $\mathbb r^4$ and trudinger--moser inequalities in $\mathbb
  r^2$.
\newblock {\em preprint, arXiv:1812.00413}, 2018.

\bibitem{CohnLu}
W.~S. Cohn and G.~Lu.
\newblock Best constants for {M}oser-{T}rudinger inequalities on the
  {H}eisenberg group.
\newblock {\em Indiana Univ. Math. J.}, 50(4):1567--1591, 2001.

\bibitem{CohnLu1}
W.~S. Cohn and G.~Z. Lu.
\newblock Best constants for {M}oser-{T}rudinger inequalities, fundamental
  solutions and one-parameter representation formulas on groups of {H}eisenberg
  type.
\newblock {\em Acta Math. Sin. (Engl. Ser.)}, 18(2):375--390, 2002.

\bibitem{deOliveira}
J.~F. de~Oliveira and J.~a.~M. do~\'{O}.
\newblock Trudinger-{M}oser type inequalities for weighted {S}obolev spaces
  involving fractional dimensions.
\newblock {\em Proc. Amer. Math. Soc.}, 142(8):2813--2828, 2014.

\bibitem{delaTorre}
A.~DelaTorre and G.~Mancini.
\newblock Improved adams--type inequalities and their extremals in dimension
  $2m$.
\newblock {\em preprint, arXiv:1711.00892}, 2017.

\bibitem{DOO}
J.~a.~M. do~\'{O} and M.~de~Souza.
\newblock A sharp inequality of {T}rudinger-{M}oser type and extremal functions
  in {$H^{1,n}(\Bbb{R}^n)$}.
\newblock {\em J. Differential Equations}, 258(11):4062--4101, 2015.

\bibitem{DongYang}
Y.~Q. Dong and Q.~H. Yang.
\newblock An interpolation of {H}ardy inequality and {M}oser-{T}rudinger
  inequality on {R}iemannian manifolds with negative curvature.
\newblock {\em Acta Math. Sin. (Engl. Ser.)}, 32(7):856--866, 2016.

\bibitem{Flucher92}
M.~Flucher.
\newblock Extremal functions for the {T}rudinger-{M}oser inequality in {$2$}
  dimensions.
\newblock {\em Comment. Math. Helv.}, 67(3):471--497, 1992.

\bibitem{FM1}
L.~Fontana and C.~Morpurgo.
\newblock Sharp exponential integrability for critical {R}iesz potentials and
  fractional {L}aplacians on {$\Bbb R^n$}.
\newblock {\em Nonlinear Anal.}, 167:85--122, 2018.

\bibitem{FM2}
L.~Fontana and C.~Morpurgo.
\newblock Adams inequalities for {R}iesz subcritical potentials.
\newblock {\em Nonlinear Anal.}, 192:111662, 32, 2020.

\bibitem{FontanaMorpurgo2020}
L.~Fontana and C.~Morpurgo.
\newblock Adams inequalities for {R}iesz subcritical potentials.
\newblock {\em Nonlinear Anal.}, 192:111662, 32, 2020.

\bibitem{Yudovic1961}
V.~I. Judovi\v{c}.
\newblock Some estimates connected with integral operators and with solutions
  of elliptic equations.
\newblock {\em Dokl. Akad. Nauk SSSR}, 138:805--808, 1961.

\bibitem{Karmakar}
D.~Karmakar and K.~Sandeep.
\newblock Adams inequality on the hyperbolic space.
\newblock {\em J. Funct. Anal.}, 270(5):1792--1817, 2016.

\bibitem{LamLuAdams}
N.~Lam and G.~Lu.
\newblock Sharp {A}dams type inequalities in {S}obolev spaces
  {$W^{m,\frac{n}{m}} (\Bbb R^n)$} for arbitrary integer {$m$}.
\newblock {\em J. Differential Equations}, 253(4):1143--1171, 2012.

\bibitem{LamLuHei}
N.~Lam and G.~Lu.
\newblock Sharp {M}oser-{T}rudinger inequality on the {H}eisenberg group at the
  critical case and applications.
\newblock {\em Adv. Math.}, 231(6):3259--3287, 2012.

\bibitem{LamLusingular}
N.~Lam and G.~Lu.
\newblock Sharp singular {A}dams inequalities in high order {S}obolev spaces.
\newblock {\em Methods Appl. Anal.}, 19(3):243--266, 2012.

\bibitem{LamLunew}
N.~Lam and G.~Lu.
\newblock A new approach to sharp {M}oser-{T}rudinger and {A}dams type
  inequalities: a rearrangement-free argument.
\newblock {\em J. Differential Equations}, 255(3):298--325, 2013.

\bibitem{LiLuYang}
J.~Li, G.~Lu, and Q.~Yang.
\newblock Fourier analysis and optimal {H}ardy-{A}dams inequalities on
  hyperbolic spaces of any even dimension.
\newblock {\em Adv. Math.}, 333:350--385, 2018.

\bibitem{LiYang}
X.~Li and Y.~Yang.
\newblock Extremal functions for singular {T}rudinger-{M}oser inequalities in
  the entire {E}uclidean space.
\newblock {\em J. Differential Equations}, 264(8):4901--4943, 2018.

\bibitem{LiRuf2008}
Y.~Li and B.~Ruf.
\newblock A sharp {T}rudinger-{M}oser type inequality for unbounded domains in
  {$\Bbb R^n$}.
\newblock {\em Indiana Univ. Math. J.}, 57(1):451--480, 2008.

\bibitem{Lin96}
K.-C. Lin.
\newblock Extremal functions for {M}oser's inequality.
\newblock {\em Trans. Amer. Math. Soc.}, 348(7):2663--2671, 1996.

\bibitem{LuTang2013}
G.~Lu and H.~Tang.
\newblock Best constants for {M}oser-{T}rudinger inequalities on high
  dimensional hyperbolic spaces.
\newblock {\em Adv. Nonlinear Stud.}, 13(4):1035--1052, 2013.

\bibitem{LuTang2016}
G.~Lu and H.~Tang.
\newblock Sharp singular {T}rudinger-{M}oser inequalities in
  {L}orentz-{S}obolev spaces.
\newblock {\em Adv. Nonlinear Stud.}, 16(3):581--601, 2016.

\bibitem{LuYangHA}
G.~Lu and Q.~Yang.
\newblock Sharp {H}ardy-{A}dams inequalities for bi-{L}aplacian on hyperbolic
  space of dimension four.
\newblock {\em Adv. Math.}, 319:567--598, 2017.

\bibitem{LuYangAiM}
G.~Lu and Y.~Yang.
\newblock Adams' inequalities for bi-{L}aplacian and extremal functions in
  dimension four.
\newblock {\em Adv. Math.}, 220(4):1135--1170, 2009.

\bibitem{LuZhu}
G.~Lu and M.~Zhu.
\newblock A sharp {T}rudinger-{M}oser type inequality involving {$L^n$} norm in
  the entire space {$\Bbb{R}^n$}.
\newblock {\em J. Differential Equations}, 267(5):3046--3082, 2019.

\bibitem{Mancini}
G.~Mancini and L.~Martinazzi.
\newblock Extremals for fractional moser--trudinger inequalities in dimension
  $1$ via harmonic extensions and commutator estimates.
\newblock {\em preprint, arXiv:1904.10267}, 2019.

\bibitem{ManciniSandeep2010}
G.~Mancini and K.~Sandeep.
\newblock Moser-{T}rudinger inequality on conformal discs.
\newblock {\em Commun. Contemp. Math.}, 12(6):1055--1068, 2010.

\bibitem{ManciniSandeepTintarev2013}
G.~Mancini, K.~Sandeep, and C.~Tintarev.
\newblock Trudinger-{M}oser inequality in the hyperbolic space {${\Bbb H}^N$}.
\newblock {\em Adv. Nonlinear Anal.}, 2(3):309--324, 2013.

\bibitem{Martinazzi}
L.~Martinazzi.
\newblock Fractional {A}dams-{M}oser-{T}rudinger type inequalities.
\newblock {\em Nonlinear Anal.}, 127:263--278, 2015.

\bibitem{Moser70}
J.~Moser.
\newblock A sharp form of an inequality by {N}. {T}rudinger.
\newblock {\em Indiana Univ. Math. J.}, 20:1077--1092, 1970/71.

\bibitem{NgoNguyenRMI}
Q.~A. Ng\^{o} and V.~H. Nguyen.
\newblock Sharp adams--moser--trudinger type inequalities in the hyperbolic
  space.
\newblock {\em to appear in Revista Matem\'atica Iberoamericana}, 2016.

\bibitem{NgoNguyenAMV}
Q.~A. Ng\^{o} and V.~H. Nguyen.
\newblock Sharp constant for {P}oincar\'{e}-type inequalities in the hyperbolic
  space.
\newblock {\em Acta Math. Vietnam.}, 44(3):781--795, 2019.

\bibitem{Nguyen4}
V.~H. Nguyen.
\newblock A sharp adams inequality in dimension four and its extremal
  functions.
\newblock {\em preprint, arXiv:1701.08249}, 2017.

\bibitem{NguyenMT2018}
V.~H. Nguyen.
\newblock Improved {M}oser-{T}rudinger type inequalities in the hyperbolic
  space {$\Bbb{H}^n$}.
\newblock {\em Nonlinear Anal.}, 168:67--80, 2018.

\bibitem{Nguyenimproved}
V.~H. Nguyen.
\newblock Improved singular {M}oser-{T}rudinger and their extremal functions.
\newblock {\em Potential Analysis, to appear.}, 2018.

\bibitem{NguyenPS2018}
V.~H. Nguyen.
\newblock The sharp {P}oincar\'{e}-{S}obolev type inequalities in the
  hyperbolic spaces {$\Bbb{H}^n$}.
\newblock {\em J. Math. Anal. Appl.}, 462(2):1570--1584, 2018.

\bibitem{NguyenCCM}
V.~H. Nguyen.
\newblock Extremal functions for the {M}oser-{T}rudinger inequality of
  {A}dimurthi-{D}ruet type in {$W^{1,N}(\Bbb R^N)$}.
\newblock {\em Commun. Contemp. Math.}, 21(4):1850023, 37, 2019.

\bibitem{NguyenHMT}
V.~H. Nguyen.
\newblock The sharp hardy-moser-trudinger inequality in dimension $n$.
\newblock {\em preprint, arXiv:1909.12587}, 2019.

\bibitem{Nguyen2020a}
V.~H. Nguyen.
\newblock The sharp {S}obolev type inequalities in the {L}orentz--{S}obolev
  spaces in the hyperbolic spaces.
\newblock {\em preprint}, 2019.

\bibitem{Nguyen2020b}
V.~H. Nguyen.
\newblock The sharp higher order {L}orentz-{P}oincar\'e and {L}orentz-{S}obolev
  inequalities in the hyperbolic spaces.
\newblock {\em preprint}, 2020.

\bibitem{NguyenLorentz}
V.~H. Nguyen.
\newblock Singular adams inequalities in {L}orentz--{S}obolev spaces.
\newblock {\em in preparation}, 2020.

\bibitem{Pohozaev1965}
S.~I. Poho\v{z}aev.
\newblock On the eigenfunctions of the equation {$\Delta u+\lambda f(u)=0$}.
\newblock {\em Dokl. Akad. Nauk SSSR}, 165:36--39, 1965.

\bibitem{Ruf2005}
B.~Ruf.
\newblock A sharp {T}rudinger-{M}oser type inequality for unbounded domains in
  {$\Bbb R^2$}.
\newblock {\em J. Funct. Anal.}, 219(2):340--367, 2005.

\bibitem{RufSani}
B.~Ruf and F.~Sani.
\newblock Sharp {A}dams-type inequalities in {$\Bbb{R}^n$}.
\newblock {\em Trans. Amer. Math. Soc.}, 365(2):645--670, 2013.

\bibitem{Tintarev2014}
C.~Tintarev.
\newblock Trudinger-{M}oser inequality with remainder terms.
\newblock {\em J. Funct. Anal.}, 266(1):55--66, 2014.

\bibitem{Trudinger67}
N.~S. Trudinger.
\newblock On imbeddings into {O}rlicz spaces and some applications.
\newblock {\em J. Math. Mech.}, 17:473--483, 1967.

\bibitem{WangYe2012}
G.~Wang and D.~Ye.
\newblock A {H}ardy-{M}oser-{T}rudinger inequality.
\newblock {\em Adv. Math.}, 230(1):294--320, 2012.

\bibitem{YangLi2019}
Q.~Yang and Y.~Li.
\newblock Trudinger-{M}oser inequalities on hyperbolic spaces under {L}orentz
  norms.
\newblock {\em J. Math. Anal. Appl.}, 472(1):1236--1252, 2019.

\bibitem{YangSuKong}
Q.~Yang, D.~Su, and Y.~Kong.
\newblock Sharp {M}oser-{T}rudinger inequalities on {R}iemannian manifolds with
  negative curvature.
\newblock {\em Ann. Mat. Pura Appl. (4)}, 195(2):459--471, 2016.

\bibitem{Yangjfa}
Y.~Yang.
\newblock A sharp form of {M}oser-{T}rudinger inequality in high dimension.
\newblock {\em J. Funct. Anal.}, 239(1):100--126, 2006.

\end{thebibliography}

\end{document}